\documentclass[12pt]{article}
\title{The pair correlation function of the
$\Sine_{6}$ process}
\usepackage{palatino}
\date{}
\author{Shengqi Qiu, Yahui Qu, Lingfan Yuan, Benedek Valk\'o,
Spencer Venancio}

\usepackage{amsmath, amsthm, amssymb}
\usepackage{graphicx}
\usepackage{amsmath, enumerate}
\usepackage{color, url}

\usepackage{hyperref}
    \newtheorem{theorem}{Theorem}
    \newtheorem{lemma}[theorem]{Lemma}
    \newtheorem{proposition}[theorem]{Proposition}

\theoremstyle{definition} 

\newcommand{\Z}{{\mathbb Z}}

\newcommand{\R}{{\mathbb R}}
\newcommand{\CC}{{\mathbb C}}

\newcommand{\lstar}{{\raise-0.15ex\hbox{$\scriptstyle \ast$}}}

\theoremstyle{remark} 
\newcommand{\Sineb}{\operatorname{Sine}_{\beta}}
\newcommand{\Sine}{\operatorname{Sine}}

\definecolor{violet}{rgb}{0.8,0,0.2}

\newcommand{\sinc}{\operatorname{sinc}}

\definecolor{darkgreen}{rgb}{.2,0.4,0.3}

\begin{document}
\maketitle

\begin{abstract}
We derive an explicit formula for the pair correlation function of the $\Sine_6$ process
in terms of Bessel functions of the first kind. This provides the first single-variable special function representation of the pair correlation function for the bulk limit of a beta-ensemble beyond the classical values of $\beta=1,2,$ and $4$. 
\end{abstract}

\section{Introduction}

The size $n$ Gaussian beta-ensemble is given by the joint density 
\begin{align}
  p_{n,\beta}(\lambda_1,\dots, \lambda_n)=\frac{1}{Z_{n,\beta}} \prod_{1\le j<k\le n} |\lambda_j-\lambda_k|^{\beta} \prod_{k=1}^n e^{-\frac{\beta}{4}\lambda_k^2},\qquad \lambda_k\in \R, \label{eq:GbE}
\end{align}
where $Z_{n,\beta}$ is an explicit constant \cite{ForBook}. The bulk scaling limit of the Gaussian $\beta$-ensemble is the $\Sineb$ process, a translation invariant point process with density $\frac{1}{2\pi}$ (see \cite{BVBV}, \cite{BVBV_sbo}). $\Sineb$ was shown to be the same as the point process limit of the related circular beta-ensemble (\cite{KS}, \cite{Nakano_2014}). The work \cite{BEY} proved that $\Sineb$ is the bulk scaling limit of a wide family of general beta-ensembles.

For $\beta=1,2,$ and $4$, the Gaussian beta-ensemble gives the joint eigenvalue density of Gaussian Orthogonal/Unitary/Symplectic Ensembles. It is a classical result that these are determinantal (for $\beta=2$) and Pfaffian (for $\beta=1$ and $4$) ensembles: the $k$-point correlation functions can be expressed in terms of $k\times k$ determinants and Pfaffians. These properties are inherited by the point process limits, which leads to explicit formulas for the $k$-point correlation functions $\rho_\beta^{(k)}(\lambda_1,\dots, \lambda_k)$ of $\Sine_\beta$ for $\beta=1,2,$ and $4$. (See e.g.~\cite{AGZ, ForBook}.) 

By translation invariance the pair correlation function satisfies $\rho_{\beta}^{(2)}(x,y)=\rho_{\beta}^{(2)}(0,y-x)$. Let $\sinc(x)=\frac{\sin(x)}{x}$. Then for the classical $\beta$ values, the pair correlation is given by 
\begin{align}
\rho_{1}^{(2)}(0,\lambda)&=\frac{1}{4\pi^2}\left(1- \sinc^2(\lambda/2) + \sinc'(\lambda/2)\left(\int_0^{\lambda/2} \sinc(t)dt-\tfrac{\pi}{2} \operatorname{sgn}(\lambda)\right)\right),\\
    \rho_{2}^{(2)}(0,\lambda)&=\frac{1}{4\pi^2}\left(1-\sinc^2(\lambda/2)\right),\\
    \rho_{4}^{(2)}(0,\lambda)& = \frac{1}{4\pi^2}\left(1- \sinc^2(\lambda) + \sinc'(\lambda)\int_0^{\lambda} \sinc(t)dt\right).
\end{align}
Our main result is an explicit formula for $ \rho_{6}^{(2)}(0,\lambda)$, the pair correlation function of the $\Sine_6$ process. Recall that $J_\nu$ is the Bessel function of the first kind. In what follows, all fractional powers $\lambda^\alpha$ are defined using the principal branch on $\CC\setminus (-\infty,0]$. 

\begin{theorem}\label{thm:main}
Introduce the functions 
\begin{align}\label{eq:def_h_b}
h(\lambda)&=e^{3 i \lambda }-\left(1+3i \lambda-\tfrac{9}{2} \lambda ^2-\tfrac{9}{2}  i \lambda ^3\right),\qquad 
b(\lambda)
=\frac{20 i \pi   }{27 } e^{-\frac{3 i \lambda }{2}} \lambda^{-5/3} h(\lambda),\\
      \label{eq:r_pm}
        r_{+}(\lambda)&=\lambda^{-1/6}(J_{1/6}(\lambda/2)-i J_{7/6}(\lambda/2)), \qquad  r_{-}(\lambda)=\lambda^{-1/6}(J_{-1/6}(\lambda/2)+i J_{-7/6}(\lambda/2)),\\ \label{eq:def-l-pm}
         \ell_\pm(\lambda)
 &=\frac{1}{40}e^{3i\lambda/2}
 \lambda^{-2}
 \Bigl(
(6i\lambda^2-6\lambda)r_\pm'(\lambda)+(3\lambda^2+9i\lambda-8)r_\pm(\lambda)
 \Bigr).
\end{align}
Then for $\lambda\in \R_+$ we have
\begin{multline}
\label{eq:mainresult}
    \rho_{6}^{(2)}(0,\lambda)=\frac{1}{4\pi^2}\Re\bigg\{1-\frac{4}{27}\lambda^{-4}h(\lambda) +   \ell_+(\lambda) \int_0^\lambda b(s) r_{-}(s) ds-\ell_-(\lambda)
\int_0^\lambda b(s) r_{+}(s) ds
    \bigg\}.
\end{multline}

\end{theorem}

Since  $\rho_6^{(2)}(0,\lambda)$ is an even function  of $\lambda$ (and equal to 0 for $\lambda=0$),  equation \eqref{eq:mainresult} determines $\rho_6^{(2)}(0,\lambda)$ for all $\lambda\in \R$. Although this may not be apparent from the formula, the function  $\rho_{6}^{(2)}(0,\lambda)$ can be extended to an entire function of $\lambda\in \CC$, see Theorem \ref{thm:correlation} below.

With Bessel function identities one can express  the functions $\ell_{+}, \ell_{-}$ in terms of the functions $J_{\pm 1/6},J_{\pm 7/6}$ (see \eqref{eq:Bessel_identities} below). Hence, the expression \eqref{eq:mainresult} can be written in terms of sine, cosine, power, and Bessel functions, using a single integration.

For $\beta=2n, n\in \Z_+$, Forrester  provided a representation for the correlation functions of all orders for the circular beta-ensemble in terms of generalized hypergeometric functions and Jack polynomials, and also evaluated the scaling limit of these functions (see \cite{Forrester_1992}, \cite{Forrester_1994}, and Chapter 13 of \cite{ForBook}). The recent results of \cite{AN2026} show that these limits are equal to the correlation functions of the $\Sineb$ process. For the pair correlation function,  Forrester's formula gives 
  \begin{align}\notag
 \rho_{2n}^{(2)}(0,\lambda)=&\frac1{4\pi^2} \cdot \frac{n^{2n} (n!)^3 }{(2n)!(3n)! } \frac{e^{-i n \lambda} \lambda^{2n}}{S_{2n}(-1+1/n,-1+1/n, 1/n)}\times\\
 &\qquad \int\limits_{[0,1]^{2n}} \prod_{j=1}^{2n} \left(e^{i \lambda u_j} u_j^{-1+1/n} (1-u_j)^{-1+1/n} \right)\prod_{j<k}|u_j-u_k|^{2/n} \prod_{j=1}^{2n}du_j .\label{eq:paircorr_For_2n}
\end{align}
Here $S_{2n}(-1+1/n,-1+1/n, 1/n)$ is the Selberg integral on $[0,1]^{2n}$ that one obtains by setting $\lambda=0$ for the integral in  \eqref{eq:paircorr_For_2n}. This integral  can be  evaluated explicitly in terms of the $\Gamma$ function. In particular, for $\beta=6$, one obtains the following expression:
\begin{align}
& \rho_{6}^{(2)}(0,\lambda)=\frac{\Gamma \left(\frac{1}{3}\right)^6}{5120 \pi ^6} e^{-3 i \lambda } \lambda ^6 \int\limits_{[0,1]^{6}} \prod_{j=1}^{6} \left(e^{i \lambda u_j} u_j^{-2/3} (1-u_j)^{-2/3} \right)\prod_{j<k}|u_j-u_k|^{2/3} \prod_{j=1}^{6}du_j.\label{eq:paircorr_For}
\end{align}
 We are not aware of a direct way to reduce \eqref{eq:paircorr_For}  to our expression in Theorem \ref{thm:main}.

Our proof relies on the results of \cite{QV2025}, where a novel characterization of $\rho_{2n}^{(2)}(0,\lambda)$ was given in terms of a complex-valued $n$-dimensional system of differential equations. When $n=3$, this leads to a third-order differential equation. We are able to reduce this equation to a second-order differential equation, which yields the explicit formula stated in  Theorem \ref{thm:main}. In Section \ref{sec:reduction} we show that a similar reduction works for all $n$:  the pair correlation function  of $\Sine_{2n}$ can be expressed 
in terms of the solution of a complex-valued ODE of order $n-1$.
Note that \cite{Forrester_20212} provides a method for constructing a real-valued ODE of order $2n+1$ with   $\rho_{2n}^{(2)}(0,\lambda)$ as a solution.

We note that \cite{QV2025} also provides a characterization of the pair correlation function of $\Sineb$ for general $\beta>0$ in terms of the characteristic function of a certain random variable. The recent results of Assiotis and Najnudel \cite{AN2026} provide a different characterization of the correlation functions of all orders for $\Sineb$ in terms of functionals of a particular random entire function. 

To our knowledge, our result is the first explicit single-variable  special function representation of the bulk pair correlation function for a non-classical $\beta$ value. 
It is interesting to note that similar results do exist in the soft edge scaling limit. 
In \cite{Rumanov_2015} Rumanov derived a representation for the Tracy-Widom$_\beta$ distribution (the scaling limit of the largest point in the Gaussian $\beta$-ensemble) for $\beta=6$ in terms of the Painlev\'e II equation. See also \cite{GIKM2016} and \cite{IP2020} for related results. 
\smallskip

The integral representation \eqref{eq:paircorr_For_2n} can be used to find the asymptotic behavior of $\rho_{2n}^{(2)}(0,\lambda)$ as $\lambda\to \infty$, see Proposition 13.13  of \cite{ForBook}. 
We record the following result which provides a minor improvement  for the $\beta=6$ case.
\begin{proposition}
\label{prop:rho6_asymptotics}
As \(\lambda\to\infty\), we have
\begin{align}
    4\pi^2 \rho_6^{(2)}(0,\lambda)=1-\frac{2}{3 \lambda^2} +\Gamma\left(\tfrac13\right)^2\, \left(\frac29 \frac{\cos \lambda}{\lambda^{2/3}}-\frac{16}{81}\frac{\sin \lambda}{\lambda^{5/3}}-\frac{64}{729} \frac{\cos \lambda}{\lambda^{8/3}} +\frac{8}{81} \frac{\cos(2\lambda)}{\lambda^{8/3}} \right)+\mathcal{O}(\lambda^{-11/3}).
    \label{eq:rho6_largelambda}
\end{align}
\end{proposition}
In Section \ref{sec:large_lambda} we provide a short outline of how these asymptotics can be derived from   \eqref{eq:paircorr_For} using  the methods of Chapter 13 of \cite{ForBook}. We  also explain how \eqref{eq:rho6_largelambda}  can be obtained  from Theorem \ref{thm:main} as well. 
The paper
\cite{QV2025} provides a decay estimate for $\rho_{\beta}^{(2)}(0,\lambda)$ for general $\beta>0$. For $\beta=6$ the provided bound is   $|4\pi^2 \rho_6^{(2)}(0,\lambda)-1|\le c \lambda^{-2/3}$ for large $\lambda$. Recently, Dumaz and Malvy \cite{DM2026}  provided decay estimates for  correlation functions of all orders for $\Sineb$ in the long-range regime.  \medskip

The integral representation \eqref{eq:paircorr_For_2n} implies that as $\lambda\to 0$, we have
\begin{align}
    \rho_{2n}^{(2)}(0,\lambda)=  \frac1{4\pi^2} \cdot \frac{n^{2n} (n!)^3 }{(2n)!(3n)! }\lambda^{2n}+\mathcal{O}(\lambda^{2n+2}).
\end{align}
The results of \cite{QV2025} provided an alternate derivation of this statement, and gave a full Taylor-series expansion for $ \rho_{2n}^{(2)}(0,\lambda)$, see Theorem \ref{thm:correlation} below. The results of \cite{AN2026} imply that for general $\beta>0$ the pair correlation function $\rho_{\beta}^{(2)}(0,\lambda)$ behaves like $c_\beta |\lambda|^\beta$ as $\lambda\to 0$, with an explicitly computable $c_\beta$.\smallskip

\noindent\textbf{Acknowledgments.} B.V.~was partially supported by  the University of Wisconsin – Madison Office of the Vice Chancellor for Research and Graduate Education with funding from the Wisconsin Alumni Research Foundation and by the National Science Foundation award DMS-2246435. The paper grew out of an undergraduate research project within the framework of the Madison Experimental Mathematics Lab (\href{https://mxm.math.wisc.edu}{MXM}) during the Spring 2026 semester. The authors would like to thank Caglar Uyanik and Grace Work, the director and associate director of MXM. The authors thank Peter Forrester for valuable comments.

\section{Preliminaries}\label{sec:prelim}

We start by reviewing of some of the results from \cite{QV2025} on the pair correlation function of the $\Sineb$ process with $\beta=2n$, $n\in \Z_{+}$.  Let $\mathbf{A}_{n}$ be an $n\times n$ tridiagonal matrix  and $\mathbf{B}_n$ an $n\times n$  diagonal matrix with nonzero entries given by 
\begin{align}\label{eq:AB}
      [\mathbf{A}_n]_{k,k}=-k^2, \quad [\mathbf{A}_n]_{k,k-1}=\frac12 k(k+n), \quad [\mathbf{A}_n]_{k,k+1}=\frac12 k(k-n),\quad [\mathbf{B}_n]_{k,k}=k.
\end{align}
Let $\mathbf{e}_n, \mathbf{f}_n, \mathbf{v}_n\in \R^n$ be defined as 
\begin{align}\label{eq:vectors}\mathbf{e}_n=[1,0,\dots,0]^\top, \quad  \mathbf{f}_n=[1,1,\dots,1]^\top,\quad
[\mathbf{v}_n]_k=(-1)^k\frac{ \binom{2n}{n+k}}{\binom{2n}{n}}, \qquad 1\le k\le n.
\end{align}

The following theorem combines the results of  Theorems 4, 14, 15 of \cite{QV2025}.

\begin{theorem}[\cite{QV2025}] \label{thm:correlation}
There is a unique vector-valued solution $\mathbf{q}_n:\CC\to \CC^n$ of the ODE system
\begin{align}\label{eq:ode_q_vec}
     \frac{n}{2} \lambda \mathbf{q}_n'(\lambda) = (i \frac{n}{2}\lambda  \mathbf{B}_n   +\mathbf{A}_n) \mathbf{q}_n(\lambda)+\frac{n+1}{2} \mathbf{e}_n, \qquad \mathbf{q}_n(0)=\mathbf{f}_n
\end{align}
which is an entire function in each coordinate. The coefficients in the Taylor series expansion   $\mathbf{q}(\lambda)=\sum_{j=0}^\infty \mathbf{s}_j \lambda^j$ are defined via the recursion 
\begin{align}\label{eq:q_coeff_1}
    \mathbf{s}_0&=-\frac{n+1}{2}\mathbf{A}_n^{-1}\mathbf{e}_n, \qquad \mathbf{s}_k=i \left(k \mathbf{I}-\tfrac{2}{n} \mathbf{A}_n\right)^{-1} \mathbf{B}_n \mathbf{s}_{k-1}, \qquad k\ge 1.
\end{align}
Moreover, 
\begin{align}
    \rho_{2n}^{(2)}(0,\lambda)=\frac{1}{4\pi^2}\left(1+2\;\mathbf{v}_n^\top\Re \mathbf{q}_n(\lambda)\right).
\end{align}
\end{theorem}

We also recall some basic properties of Bessel functions of the first kind. (See e.g.~\cite{AbSt}.)
The Bessel function of the first kind with parameter $\nu\in \CC$ is given by  the series 
\begin{align}\label{eq:Bessel_series}
     J_\nu(\lambda)=\sum_{k=0}^\infty \frac{(-1)^k}{k! \Gamma(k+\nu+1)} (\lambda/2)^{2k+\nu}.
\end{align}
For $\nu=n\in \Z_{<0}$, one has to take the appropriate limits of the coefficients, and we get $J_{n}(\lambda)=(-1)^n J_{-n}(\lambda)$. The function $\lambda^{-\nu} J_{\nu}(\lambda)$ can be extended as an entire function on $\CC$.

If $\nu\notin \Z$ then $J_\nu$ and $J_{-\nu}$ are linearly independent solutions of  Bessel's equation:
\begin{align}\label{eq:ODE_Bessel}
   \lambda^2  u''(\lambda)+\lambda u'(\lambda)+(\lambda^2-\nu^2) u(\lambda)=0.
\end{align}
We record the following Bessel function identities:
\begin{align}\label{eq:Bessel_identities}
        J'_\nu(\lambda)=J_{\nu-1}(\lambda)-\frac{\nu}{\lambda}J_\nu(\lambda)&=-J_{\nu+1}(\lambda)+\frac{\nu}{\lambda}J_\nu(\lambda),
\\
\label{eq:Bessel_Wr}
 J_\nu(\lambda)J'_{-\nu}(\lambda)-J_{-\nu}(\lambda)J'_{\nu}(\lambda)&=-\frac{2\sin(\pi \nu)}{\pi \lambda}.
\end{align}

\section{Proof of Theorem \ref{thm:main}}

\begin{proof}[Proof of Theorem \ref{thm:main}]
We start by applying Theorem \ref{thm:correlation} for $n=3$. We write $\mathbf{q}_3=\mathbf{q}=[q_1, q_2, q_3]^\top$. Then the ODE \eqref{eq:ode_q_vec} can be rewritten as the system 
\begin{align}\notag
    \tfrac{3}{2} \lambda  q_1'(\lambda )&=\left(-1+\tfrac{3}{2} i \lambda\right) q_1(\lambda
   )-q_2(\lambda )+2, \\ \label{eq:q_ODE6}
 \tfrac{3}{2} \lambda  q_2'(\lambda )&=5 q_1(\lambda )+(-4+3 i \lambda ) q_2(\lambda
   )-q_3(\lambda ), \\
 \tfrac{3}{2} \lambda  q_3'(\lambda )&=9 q_2(\lambda )+\left(-9+\tfrac{9}{2} i \lambda
\right) q_3(\lambda ),\notag
\end{align}
with $q_1(0)=q_2(0)=q_3(0)=1$. With $
     \mathbf{v}_3=\mathbf{v}=\left[-\tfrac{3}{4},\tfrac{3}{10},-\tfrac{1}{20}\right]^\top$, we have
    \begin{align}\label{eq:rho6}
           \rho_{6}^{(2)}(0,\lambda)=\frac{1}{4\pi^2}\Re\left(1+2\mathbf{v}^\top\mathbf{q}(\lambda)\right)=\frac{1}{4\pi^2} \Re \left\{1-\tfrac{3}{2}  q_1(\lambda )+\tfrac{3}{5} q_2(\lambda )-\tfrac{1}{10}q_3(\lambda )\right\}.
    \end{align}
From the third and second equations of the system \eqref{eq:q_ODE6} we can express $q_2, q_2', q_1,$  and $q_1'$ in terms of $q_3$ as follows: 
\begin{align}\label{eq:q2_q3}
    q_2(\lambda)&=\tfrac{1}{6} \lambda  q_3'(\lambda )+\left(1-\tfrac{1}{2} i \lambda \right) q_3(\lambda )\\
    q_2'(\lambda)&=\tfrac{1}{6} \lambda  q_3''(\lambda )+\left(\tfrac{7}{6}-\tfrac12 i \lambda \right) q_3'(\lambda )-\tfrac{1}{2} i q_3(\lambda ).
\\
\label{eq:q1_q3}
    q_1(\lambda)&=\tfrac{1}{20} \lambda ^2 q_3''(\lambda )+\left(\tfrac{29  }{60}\lambda-\tfrac14 i \lambda ^2\right) q_3'(\lambda )+\left(-\tfrac{3 }{10}\lambda ^2-\tfrac{23 }{20} i \lambda+1\right) q_3(\lambda ),
\\\notag
    q_1'(\lambda)&=\tfrac{1}{20} \lambda ^2q_3'''(\lambda )+\left(\tfrac{7 }{12} \lambda-\tfrac14 {i \lambda ^2}\right)q_3''(\lambda )+\left(-\tfrac{3 }{10}\lambda ^2-\tfrac{33 }{20}i \lambda +\tfrac{89}{60}\right)q_3'(\lambda )+\left(-\tfrac{3 }{5} \lambda-\tfrac{23}{20}  i\right)q_3(\lambda ).
\end{align}
Substituting the expressions for $q_2, q_1, q_1'$ into the first equation of \eqref{eq:q_ODE6} gives the following third-order ODE for $q_3$:
\begin{align}\label{eq:ODE_q3}
  80&=3 \lambda ^3 q_3'''(\lambda )+\left(37 \lambda ^2-18 i \lambda ^3\right) q_3''(\lambda )+\\&\qquad\left(-33 \lambda ^3-138 i \lambda ^2+115 \lambda \right) q_3'(\lambda )+\left(18 i \lambda ^3-117 \lambda ^2-195 i \lambda +80\right) q_3(\lambda ).\notag
\end{align}
From Theorem \ref{thm:correlation} (using the series representation of $\mathbf{q}(\lambda)$) we  have 
\begin{align*}
    q_3(0)=1, \qquad q_3'(0)=i, \qquad q_3''(0)=-\tfrac{9}{8}.
\end{align*}
Once we have an expression for $q_3$, the equations \eqref{eq:rho6}, \eqref{eq:q2_q3},  \eqref{eq:q1_q3} will lead to an expression for $\rho_{6}^{(2)}(0,\lambda)$.
The key observation in solving the ODE \eqref{eq:ODE_q3} is that it can be reduced to a second-order differential equation. More specifically, we  show that $q_3$ solves 
\begin{align}\label{eq:q3_A}
    3 \lambda ^2 q_3''(\lambda )-(9 i \lambda^2 -19\lambda)   q_3'(\lambda )-\left(6 \lambda ^2+27 i \lambda -20\right) q_3(\lambda )=\tfrac{160}{27 } \lambda^{-4} h(\lambda),
\end{align}
where $h$ is defined in \eqref{eq:def_h_b}, and the right side at $\lambda=0$ is defined as the appropriate limit. Denote the left side of \eqref{eq:q3_A} by $w(\lambda)$. Direct differentiation shows that $\lambda w'(\lambda)+(4-3 i \lambda ) w(\lambda )$ is exactly equal to the expression on the right side of \eqref{eq:ODE_q3}. We also have $w(0)=20q_3(0)=20$. The reduction now follows from the fact that the solution of the ODE 
\begin{align*}
    \lambda w'(\lambda)+(4-3 i \lambda ) w(\lambda )=80, \qquad w(0)=20
\end{align*}
is given by  $w(\lambda)=\tfrac{160}{27 } \lambda^{-4} h(\lambda)$. In Section \ref{sec:reduction} we will  show that a similar reduction holds for the order $n$ ODE associated to $\rho_{2n}^{(2)}(0,\lambda)$ for $n\in \Z_+$.

Introducing $r(\lambda)= \lambda^2 e^{-\frac{3}{2}i\lambda} q_3(\lambda)$, the equation \eqref{eq:q3_A} transforms to 
\begin{align}
   \lambda  r''(\lambda )+\tfrac{7 }{3}r'(\lambda )+\left(\tfrac{\lambda }{4}+\tfrac{i}{2}\right) r(\lambda )= \tfrac{160}{81} e^{-\frac{3}{2}  i \lambda } \lambda^{-3} h(\lambda ), \label{eq:ODE_r}
\end{align}
valid for $\lambda\neq 0$.
Note that $r(\lambda)$ is an entire function with $r(0)=r'(0)=0$.

Consider the homogeneous version of \eqref{eq:ODE_r}, i.e.
\begin{align}\label{eq:ODE_r_hom}
     \lambda  y''(\lambda )+\tfrac{7 }{3}y'(\lambda )+\left(\tfrac{\lambda }{4}+\tfrac{i}{2}\right) y(\lambda )=0.
\end{align}
Direct differentiation, together with \eqref{eq:Bessel_identities}, shows that the ODE \eqref{eq:ODE_r_hom} is solved   by the functions $r_+$ and $r_-$ from   \eqref{eq:r_pm} on  $\lambda\in \CC\setminus (-\infty,0]$. From \eqref{eq:Bessel_Wr} it follows that the Wronskian of these functions is nonzero: 
\begin{align}
W(\lambda):=r_+(\lambda) r'_{-}(\lambda)-r_{-}(\lambda) r'_{+}(\lambda)=\frac{8 i}{3 \pi } \lambda^{-7/3},
\end{align}
hence $r_+$ and $r_-$ are linearly independent.
From \eqref{eq:Bessel_series} and the definitions of $r_{+}, r_{-}, h$ we have the following asymptotics near 0:
\begin{align}\label{eq:asympt_r_h}
    r_{+}(\lambda)=\tfrac{2^{-1/3}}{ \Gamma(7/6)}+\mathcal{O}(\lambda), \quad
r_{-}(\lambda)=- \tfrac{2^{7/3} i}{\Gamma(-1/6)} \lambda^{-4/3}+\mathcal{O}(\lambda^{-1/3}),\quad 
h(\lambda)=\tfrac{27}{8}\lambda^4+\mathcal{O}(\lambda^5).
\end{align}
Using variation of parameters, it follows that  all solutions of 
  \eqref{eq:ODE_r} (hence $r(\lambda)$ as well) can be written as 
\begin{align*}
        c_+r_+(\lambda) + c_- r_-(\lambda) + r_+(\lambda)\int_0^\lambda b(s)  r_{-}(s) ds- r_-(\lambda)\int_0^\lambda b(s)  r_{+}(s) ds.
\end{align*}
with $c_+, c_{-}\in \CC$ and  $b(\lambda)=-\frac{1}{\lambda W(\lambda)}
     \cdot \frac{160}{81} e^{-\frac{3}{2}  i \lambda } \lambda^{-3} h(\lambda )$, agreeing with the definition given in \eqref{eq:def_h_b}. 
  The asymptotics \eqref{eq:asympt_r_h} together with the fact that $r(0)=r'(0)=0$  imply that
\begin{align}\label{eq:r}
    r(\lambda)   &= r_+(\lambda)\int_0^\lambda b(s)  r_{-}(s) ds- r_-(\lambda)\int_0^\lambda b(s)  r_{+}(s) ds,
\end{align}
and that the integrals in \eqref{eq:r} are well-defined. 
From $q_3(\lambda)=\lambda^{-2}  e^{\frac{3}{2}i\lambda} r(\lambda)$, we get 
\begin{align}\label{eq:q3_2}
    q_3(\lambda)&=
    \lambda^{-2}  e^{\frac{3}{2}i\lambda}r_+(\lambda)\int_0^\lambda b(s)  r_{-}(s) ds- \lambda^{-2}  e^{\frac{3}{2}i\lambda}r_-(\lambda)\int_0^\lambda b(s)  r_{+}(s) ds.
\end{align}
This formula holds for $\lambda\in \CC\setminus (-\infty,0]$, in particular for $\lambda\in \R_+$. From \eqref{eq:q2_q3} and \eqref{eq:q1_q3} we get
\begin{align}\notag
    2\mathbf{v}^\top \mathbf{q}(\lambda)&=-\tfrac{3}{2}  q_1(\lambda )+\tfrac{3}{5} q_2(\lambda )-\tfrac{1}{10}q_3(\lambda )\\
    &=-\tfrac{3}{40}  \lambda ^2 q_3''(\lambda )+\left(-\tfrac{5 }{8} \lambda+\tfrac{3 i }{8}\lambda ^2\right) q_3'(\lambda )+\left(\tfrac{9 }{20}\lambda ^2+\tfrac{57 i }{40}\lambda -1\right) q_3(\lambda ).
\end{align}
Using \eqref{eq:q3_A} we can further reduce this to 
\begin{align}\label{eq:qv_2}
   2\mathbf{v}^\top \mathbf{q}(\lambda) =-\tfrac{4 }{27 }\lambda^{-4} h(\lambda )+\left(-\tfrac{3  }{20}\lambda+\tfrac{3 i }{20}\lambda ^2\right) q_3'(\lambda )+\left(\tfrac{3 }{10}\lambda ^2+\tfrac{3 i  }{4}\lambda-\tfrac{1}{2}\right) q_3(\lambda ).
\end{align}
Differentiating \eqref{eq:q3_2} yields
\begin{align*}
    q_3'(\lambda)=&\tfrac1{2} \lambda^{-3} e^{\frac{3 i \lambda }{2}} \left(2 \lambda  r_+'(\lambda )+(-4+3 i \lambda ) r_+(\lambda )\right)\int_0^\lambda b(s)  r_{-}(s) ds\\
    &\qquad -\tfrac1{2} \lambda^{-3} e^{\frac{3 i \lambda }{2}} \left(2 \lambda  r_-'(\lambda )+(-4+3 i \lambda ) r_-(\lambda )\right)\int_0^\lambda b(s)  r_{+}(s) ds.
\end{align*}
Together with \eqref{eq:q3_2} and \eqref{eq:qv_2} we obtain
\begin{align}
    2\mathbf{v}^\top \mathbf{q}(\lambda) =& -\frac{4}{27}\lambda^{-4}h(\lambda) +   \ell_+(\lambda) \int_0^\lambda b(s) r_{-}(s) ds-\ell_-(\lambda)
\int_0^\lambda b(s) r_{+}(s) ds,
\end{align}
with $\ell_{\pm} $ defined in \eqref{eq:def-l-pm}.  The statement of Theorem \ref{thm:main} follows from \eqref{eq:rho6}.\end{proof}

\section{Reduction of the ODE \eqref{eq:ode_q_vec}}\label{sec:reduction}

One of the key steps in the proof of Theorem \ref{thm:main} is the reduction of the third-order ODE \eqref{eq:ODE_q3} to the second-order equation \eqref{eq:q3_A}. We  show that such a reduction holds in general for $\beta=2n$. 
The key is the following lemma. 

\begin{lemma}\label{lem:reduction} Fix $n\in \Z_+$, and 
let $\mathbf{A}_n, \mathbf{B}_n$ be defined as in \eqref{eq:AB}. There is a vector-valued polynomial function $\mathbf{y}_n(\lambda):\CC\to \CC^n$ so that with ${\mathbf{z}}_n(\lambda)=e^{-i n\lambda} \lambda^{n+1}\mathbf{y}_n(\lambda)$  we have
 \begin{align}\label{eq:ODE_z}
     \tfrac{n}{2} \lambda {\mathbf{z}}_n'(\lambda) = -(i \tfrac{n}{2}\lambda  \mathbf{B}_n   +\mathbf{A}_n^\top) {\mathbf{z}}_n(\lambda).
\end{align}
\end{lemma}

Let $\mathbf{q}_n(\lambda)=[q_1(\lambda), \dots, q_n(\lambda)]^\top$ be the solution of the ODE system \eqref{eq:ode_q_vec} from Theorem \ref{thm:correlation}. Using Lemma \ref{lem:reduction} we  show that $\mathbf{q}_n(\lambda)$ (and hence $\rho_{2n}^{(2)}(0,\lambda)$) can be expressed in terms of a complex-valued ODE of order $n-1$.
As in the first steps of the proof of Theorem \ref{thm:main}, we can express the functions $q_1(\lambda),\dots, q_{n-1}(\lambda)$ in terms of $q_n(\lambda), q_n'(\lambda), \dots, q_n^{(n-1)}(\lambda)$ from the last $n-1$ equations of the ODE system \eqref{eq:ode_q_vec}. Let $\mathbf{z}_n$  be the vector-valued function from  Lemma \ref{lem:reduction}. We have
\begin{align*}
   {\tfrac{n}{2}\lambda}(\mathbf{z}_n^\top(\lambda) \mathbf{q}_n(\lambda))'&=\mathbf{z}_n^\top(\lambda)\left((i \tfrac{n}{2}\lambda  \mathbf{B}_n   +\mathbf{A}_n) \mathbf{q}_n(\lambda)+\tfrac{n+1}{2} \mathbf{e}_n\right) - \left((i \tfrac{n}{2}\lambda  \mathbf{B}_n   +\mathbf{A}_n^\top) {\mathbf{z}}_n(\lambda)\right)^\top \mathbf{q}_n(\lambda)\\&=\tfrac{n+1}{2}\mathbf{z}_n^\top(\lambda) \mathbf{e}_n,
\end{align*}
and 
\begin{align}\label{eq:zq}
     \mathbf{z}_n^\top(\lambda) \mathbf{q}_n(\lambda)={\tfrac{n+1}{n} }\int_0^\lambda s^{-1} {\mathbf{z}_n^\top(s) \mathbf{e}_n} ds.
\end{align}
Since ${\mathbf{z}}_n(\lambda)=e^{-i n\lambda} \lambda^{n+1}\mathbf{y}_n(\lambda)$ with $\mathbf{y}_n$ a polynomial function, the integral is well-defined. 
Substitute our expressions for $q_1(\lambda),\dots, q_{n-1}(\lambda)$ in terms of $q_n(\lambda), q_n'(\lambda), \dots, q_n^{(n-1)}(\lambda)$ into \eqref{eq:zq}. This leads to the advertised differential equation of order $n-1$ for $q_n$. The coefficient of $q_n^{(n-1)}(\lambda)$ in the equation is nonzero, hence the order is exactly $n-1$. 
We have $q_n(0)=1$, and the Taylor expansion \eqref{eq:q_coeff_1} determines the derivatives of $q_n$ at $0$. 

\begin{proof}[Proof of Lemma \ref{lem:reduction}]
Let $L^{(\alpha)}_k(x), k\ge 0$ denote the generalized Laguerre polynomials corresponding to parameter $\alpha$. For $\alpha>-1$, these are the orthogonal polynomials with respect to the measure $e^{-x} x^\alpha dx$ on $[0,\infty)$. These polynomials can  also be defined via the  recursion
\begin{align}
    &L^{(\alpha)}_0(x)=1, \qquad  L^{(\alpha)}_1(x)=1+\alpha-x\\
    &(k+1)L^{(\alpha)}_{k+1}(x)=(2k+1+\alpha-x)L_k^{(\alpha)}(x)-(k+\alpha)L_{k-1}^{(\alpha)}(x), \quad k\ge 1. \label{eq:Lagg_rec}
\end{align}
They satisfy the identity
\begin{align}\label{eq:Lag_diff}
    x \, \frac{d}{dx} L_k^{(\alpha)}(x)=k L_k^{(\alpha)}(x)-(k+\alpha)L_{k-1}^{(\alpha)}(x), \qquad k\ge 1,
\end{align}
(see e.g.~\cite{AbSt}). Introduce the notation $
    (x)^{\uparrow k}=\prod_{j=0}^{k-1}(x+j)$ for the Pochhammer symbol, 
with the empty product defined as 1. We  show that with 
\begin{align}\label{eq:y_Lagg}
    y_k(\lambda)=\frac{(1-n)^{\uparrow k-1}}{k \,  (2+n)^{\uparrow k-1}} \, L_{k-1}^{(n+1)}(i n \lambda), \qquad 1\le k\le n
\end{align}
the vector-valued polynomial function $\mathbf{y}_n(\lambda)=[y_1(\lambda), \dots, y_n(\lambda)]^\top$ satisfies the requirements of Lemma \ref{lem:reduction}.
Note that we can extend the definition to $k=0$ and $k=n+1$, with $y_0(\lambda)=y_{n+1}(\lambda)=0$. 

Set ${\mathbf{z}}_n(\lambda)=e^{-i n\lambda} \lambda^{n+1}\mathbf{y}_n(\lambda)$, our goal is to prove that \eqref{eq:ODE_z} holds. For this, we  need to show that for $1\le k\le n$ we have
\begin{align}\notag
  0=& n\lambda y_k'(\lambda)
+
\left(n(n+1)-2k^2+in\lambda(k-n)\right)y_k(\lambda)\\
&\quad
+
(k-1)(k-1-n)y_{k-1}(\lambda)
+
(k+1)(k+1+n)y_{k+1}(\lambda)\label{eq:y_zero}
\end{align}
Using \eqref{eq:Lagg_rec} and \eqref{eq:y_Lagg} we get
\begin{align}
   (k+1)(n+k+1)y_{k+1}(\lambda)=(k-n)\left((2k+n-i n \lambda)y_k(\lambda)+(n-k+1)\frac{k-1}{k}y_{k-1}(\lambda)\right). \label{eq:y_first}
\end{align}
From \eqref{eq:Lag_diff} and \eqref{eq:y_Lagg} we obtain
\begin{align}\notag
    \lambda y_k'(\lambda)&=i n \lambda \frac{(1-n)^{\uparrow k-1}}{k \,  (2+n)^{\uparrow k-1}} \, \left(L_{k-1}^{(n+1)}\right)'(i n \lambda)\\ \notag
    &=\frac{(1-n)^{\uparrow k-1}}{k \,  (2+n)^{\uparrow k-1}} \, \left((k-1) L_{k-1}^{(n+1)}(i n\lambda)-(k+n)L_{k-2}^{(n+1)}(in \lambda)\right)\\
    &=(k-1) y_k(\lambda)+(n-k+1)\frac{k-1}{k}  y_{k-1}(\lambda).\label{eq:y_second}
\end{align}
The identities \eqref{eq:y_first} and \eqref{eq:y_second} together imply \eqref{eq:y_zero}, finishing the proof.
\end{proof}

\section{Large $\lambda$ asymptotics}
\label{sec:large_lambda}

The statement of Proposition \ref{prop:rho6_asymptotics} can be obtained with a minor modification of the arguments presented in Section 13.2.3 of \cite{ForBook}, which we   outline here. The expression \eqref{eq:paircorr_For} for $\rho_6^{(2)}(0,\lambda)$ is given by a multidimensional Fourier integral, so its asymptotics as $\lambda\to \infty$ can be obtained using the classical methods of oscillatory integrals (see e.g.~\cite{Wong2001}). As \cite{ForBook} argues, as $\lambda\to \infty$ the main contributions to the integral in \eqref{eq:paircorr_For} come from the neighborhoods of the vertices of the cube $[0,1]^6$. By symmetry, it suffices to identify the  contributions near the vertices $\mathbf{x}_\ell, 0\le \ell\le 3$, where $\mathbf{x}_\ell\in \R^6$ has  $3-\ell$ coordinates that are equal to 0 and $3+\ell$ coordinates that are equal to 1. Section 13.2.3 of \cite{ForBook} shows that the
contributions  of $\mathbf{x}_\ell$ for  $0\le \ell \le 3$ are 
\begin{align}
     c_\ell \, \lambda^{-2 \ell^2/3}\left(\cos (\ell \lambda)+d_\ell \frac{\sin (\ell \lambda)}{\lambda}+\mathcal{O}(\lambda^{-2})\right), 
\end{align}
where  
\begin{align*}
   c_0=1, \qquad d_0=0, \qquad c_1=\frac{2}{9}\Gamma(1/3)^2 , \qquad d_1=-\frac{8}{9}, \\ c_2=\frac{8}{81}\Gamma(1/3)^2, \qquad d_2=\frac{8}{9}, \qquad c_3=\frac{8}{81},\qquad d_3=8.
\end{align*}
Exercise 2 of Section 13.2 of \cite{ForBook} outlines how to obtain the next order term for the contribution of $\mathbf{x}_0$, this identifies the non-oscillatory terms in \eqref{eq:rho6_largelambda} as $1-\tfrac{2}{3\lambda^2}$. To obtain Proposition \ref{prop:rho6_asymptotics}, one needs to identify the next order term in the contribution of $\mathbf{x}_1$ as well, which  can be done in a similar way. The resulting term is $-\Gamma(\tfrac{1}{3})^2 \tfrac{64}{729} \tfrac{\cos \lambda}{\lambda^{8/3}}$ with an error term $\mathcal{O}(\lambda^{-11/3})$, which
completes the proof of Proposition \ref{prop:rho6_asymptotics}. 

We note that  Proposition \ref{prop:rho6_asymptotics} can also be obtained directly from Theorem \ref{thm:main}. Using the Hankel asymptotic expansion of the Bessel function $J_\nu$ (see e.g.~\cite{AbSt}), we can find  approximations of $\ell_{+}(\lambda)$ and $\ell_{-}(\lambda)$ in terms of oscillating terms of orders $\lambda^{-2/3}, \lambda^{-5/3}, \lambda^{-8/3}$ with an error term of $\mathcal{O}(\lambda^{-11/3})$. Using the integral representation 
\begin{equation}
 J_\nu(z)
 =\frac{(z/2)^\nu}{\sqrt\pi\,
 \Gamma(\nu+\frac12)}
 \int_{-1}^1e^{izt}(1-t^2)^{\nu-1/2}\,dt
 \label{eq:poisson-bessel}
\end{equation}
(valid for $\nu>-1/2$)
and Fubini's theorem, we can express $\int_0^\lambda b(s) r_{\pm}(s) ds$ as  one-dimensional oscillatory integrals involving rational and exponential functions, and the exponential integral function $\operatorname{Ei}$.  The asymptotics given in \eqref{eq:rho6_largelambda} can now be obtained by analyzing the contributions of the two end-points in the resulting one-dimensional integral, similar to the multidimensional analysis outlined above.  

Both of the outlined approaches can be used to identify further terms in the large $\lambda$ asymptotics of $\rho_6^{(2)}(0,\lambda)$.


\begin{thebibliography}{10}



\bibitem{AbSt}
Milton {Abramowitz} and Irene~A. {Stegun}.
\newblock {\em {Handbook of mathematical functions with formulas, graphs, and
  mathematical tables.}}
\newblock Wiley, 1984.

\bibitem{AGZ}
Greg~W. Anderson, Alice Guionnet, and Ofer Zeitouni.
\newblock {\em Introduction to random matrices}.
\newblock Cambridge University Press, 2009.

\bibitem{AN2026}
Theodoros Assiotis and Joseph Najnudel.
\newblock Moments of {C$\beta$E} field partition function,
  $\operatorname{Sine}_{\beta}$ correlations and stochastic zeta.
\newblock \url{https://arxiv.org/abs/2602.08739}, 2026.

\bibitem{BEY}
Paul Bourgade, L{{\'a}}szl{{\'o}} Erd{\H{o}}s, and Horng-Tzer Yau.
\newblock Universality of general {$\beta$}-ensembles.
\newblock {\em Duke Math. J.}, 163(6):1127--1190, 2014.

\bibitem{DM2026}
Laure Dumaz and Martin Malvy.
\newblock {Long-Range} {C}orrelation of the {Sine}$_\beta$ point {P}rocess.
\newblock \url{https://arxiv.org/abs/2603.15289}, 2026.

\bibitem{Forrester_1992}
Peter~J. Forrester.
\newblock Selberg correlation integrals and the $1/r^2$ quantum many-body
  system.
\newblock {\em Nuclear Physics B}, 388(3):671--699, 1992.

\bibitem{Forrester_1994}
Peter~J. Forrester.
\newblock Addendum to `{S}elberg correlation integrals and the $1/r^2$ quantum
  many body system'.
\newblock {\em Nuclear Physics B}, 416(1):377--385, 1994.

\bibitem{ForBook}
Peter~J. Forrester.
\newblock {\em Log-gases and random matrices}, volume~34 of {\em London
  Mathematical Society Monographs Series}.
\newblock Princeton University Press, Princeton, NJ, 2010.

\bibitem{Forrester_20212}
Peter~J. Forrester.
\newblock Differential identities for the structure function of some random
  matrix ensembles.
\newblock {\em Journal of Statistical Physics}, 183(2), May 2021.

\bibitem{GIKM2016}
Tamara Grava, Alexander Its, Andrei Kapaev, and Francesco Mezzadri.
\newblock On the {Tracy-Widom}$_\beta$ distribution for $\beta=6$.
\newblock {\em Symmetry, Integrability and Geometry: Methods and Applications},
  November 2016.

\bibitem{IP2020}
Alexander Its and Andrei Prokhorov.
\newblock On the $\beta=6$ {Tracy-Widom} distribution and the second
  {Calogero-Painlev{\'e}} system.
\newblock \url{https://arxiv.org/abs/2010.06733}, 2020.

\bibitem{KS}
Rowan Killip and Mihai Stoiciu.
\newblock Eigenvalue statistics for {CMV} matrices: from {P}oisson to clock via
  random matrix ensembles.
\newblock {\em Duke Math. J.}, 146(3):361--399, 2009.

\bibitem{Nakano_2014}
Fumihiko Nakano.
\newblock Level statistics for one-dimensional {S}chrödinger operators and
  {G}aussian beta ensemble.
\newblock {\em Journal of Statistical Physics}, 156(1):66–93, April 2014.

\bibitem{QV2025}
Yahui Qu and Benedek Valk\'{o}.
\newblock On the pair correlation function of the {S}ine$_\beta$ process.
\newblock \url{https://arxiv.org/abs/2509.15446}, 2025.

\bibitem{Rumanov_2015}
Igor Rumanov.
\newblock Painlev{\'e} representation of {Tracy--Widom}${_\beta}$ distribution
  for $\beta = 6$.
\newblock {\em Communications in Mathematical Physics}, 342(3):843--868,
  November 2015.

\bibitem{BVBV}
Benedek Valk\'o and B\'alint Vir\'ag.
\newblock Continuum limits of random matrices and the {B}rownian carousel.
\newblock {\em Inventiones Math.}, 177:463--508, 2009.

\bibitem{BVBV_sbo}
Benedek Valk{\'o} and B{\'a}lint Vir{\'a}g.
\newblock The {Sine$_\beta$} operator.
\newblock {\em Inventiones mathematicae}, 209(1):275--327, 2017.

\bibitem{Wong2001}
R.~Wong.
\newblock {\em Asymptotic approximations of integrals}, volume~34 of {\em
  Classics in Applied Mathematics}.
\newblock Society for Industrial and Applied Mathematics (SIAM), Philadelphia,
  PA, 2001.

\end{thebibliography}

\def\cprime{$'$}

{{  
\bigskip
  \footnotesize

\noindent Shengqi Qiu, \textsc{Department of Mathematics, University of Wisconsin -- Madison, 480 Lincoln Dr,
Madison, WI 53706},  \texttt{sqiu53@wisc.edu}
\medskip

\noindent Yahui Qu, \textsc{Department of Mathematics, University of Wisconsin -- Madison, 480 Lincoln Dr,
Madison, WI 53706},  \texttt{yqu45@wisc.edu}

  \medskip

\noindent Benedek Valk\'o, \textsc{Department of Mathematics, University of Wisconsin -- Madison, 480 Lincoln Dr,
Madison, WI 53706}, \texttt{valko@math.wisc.edu}
  
  \medskip
  
\noindent Lingfan Yuan, \textsc{Department of Mathematics, University of Wisconsin -- Madison, 480 Lincoln Dr,
Madison, WI 53706},  \texttt{lyuan57@wisc.edu}

  \medskip

\noindent Spencer Venancio, \textsc{Department of Mathematics, University of Wisconsin -- Madison, 480 Lincoln Dr,
Madison, WI 53706},  \texttt{svenancio@wisc.edu}

}}

\end{document}